\documentclass{amsart}
\usepackage{color}
\usepackage{graphicx,epsf}
\usepackage{amsmath,amscd}
\newtheorem{theorem}{Theorem}[section]
\newtheorem{lemma}[theorem]{Lemma}
\newtheorem{corollary}[theorem]{Corollary}

\theoremstyle{definition}
\newtheorem{definition}[theorem]{Definition}
\newtheorem{example}[theorem]{Example}

\theoremstyle{remark}
\newtheorem{remark}[theorem]{Remark}
\numberwithin{equation}{section}

\newcommand{\Real}{{\mathbb R}}

\newcommand{\eps}{\varepsilon}

\newcommand {\hide}[1]{}

\begin{document}
\title[Lecture Notes on QE]{Lecture Notes on Complexity of Quantifier Elimination
over the Reals}

\author{Nicolai Vorobjov}
\address{
Department of Computer Science, University of Bath, Bath
BA2 7AY, England, UK}
\email{masnnv@bath.ac.uk}
\maketitle

\section*{Introduction}

These are lecture notes for a course I gave in mid-1990s for MSc students at the University of Bath.
It presents an algorithm with singly exponential complexity for the existential theory of the reals,
in the spirit of J. Renegar \cite{R}.
Some ideas are borrowed from \cite{GV} and \cite{G}.
The aim was to convey the main underlying ideas, so many of the proofs and finer details of algorithms
are either missing or just sketched.
Full proofs and algorithms can be found in the aforementioned papers or, in a somewhat different form,
in the fundamental monograph \cite{BPR}.

I changed nothing in the original notes except adding references, bibliography, and correcting obvious typos.

\section{Semialgebraic sets}

\subsection{Objectives of the course}
\medskip

We are going to study algorithms for solving systems of polynomial equations
and inequalities.
The emphasis is on the {\it complexity}, that is, on the running time of algorithms.
Elementary procedures, like Gr\"obner basis or cylindrical decomposition,
have high complexity.
Effective algorithms require more advanced mathematical tools.

Computational problems we will be dealing with go beyond just solving systems.
The most general problem we are going to consider is {\em quantifier elimination
in the first order theory of the reals}.

\subsection{Formulas and semialgebraic sets}

Let $\Real$ denote the field of all real numbers.

Suppose that we are interested in deciding whether or not a system of inequalities
$$f_1 (X_1, \ldots , X_n) \ge 0 , \ldots , f_k (X_1, \ldots , X_n) \ge 0$$
is consistent, that is, has a solution in $\Real^n$.
Here $f_i \in \Real[X_1, \ldots , X_n],\> 1 \le i \le k$ are polynomials
with real coefficients in variables $X_1, \ldots , X_n$.

We can express the statement that the system is consistent by writing
\begin{equation}\label{eq:system}
\exists X_1 \exists X_2 \cdots \exists X_n \left(
f_1 (X_1, \ldots , X_n) \ge 0 , \ldots , f_k (X_1, \ldots , X_n) \ge 0 \right).
\end{equation}

This statement is either true (if there is a real solution to the system) or false.

We can generalize formula (\ref{eq:system}) in two natural ways.
Firstly, we can consider formulae in which not all the variables are bound by quantifier $\exists$.
\medskip

\begin{example}
\begin{equation}\label{eq:ex1}
\exists X_{n_1} \exists X_{n_1+1} \cdots \exists X_n \bigl(
f_1 (X_1, \ldots , X_n) \ge 0 , \ldots , f_k (X_1, \ldots , X_n) \ge 0 \bigr)
\end{equation}
for a certain $n_1,\> 1 < n_1 \le n$.

In (\ref{eq:ex1}) variables
$$X_{n_1}, \ldots , X_n$$
are called {\em bound}, while the remaining variables
$$X_1, \ldots , X_{n_1-1}$$
are {\em free}.
\end{example}

Secondly, we can consider formulae with both types of quantifiers,
$\forall$ as well as $\exists$.

\begin{example}
A formula with the prefix
$$\forall X_{n_1} \forall X_{n_1+1} \exists X_{n_1+2} \forall X_{n_1+3} \ldots$$
has free variables $X_1, \ldots , X_{n_1-1}$ and both types of quantifiers.
Formulae of such kind are called {\em formulae of the first order theory of
$\Real$ (or of the reals)}.
\end{example}

Formulae of the first order theory of the reals can be used to express a great
variety of statements, properties, problems about real numbers, and
therefore about the geometry of $\Real^n$.

Formulae without free variables express statements (true or false), while
the ones with free variables describe {\em properties}.

A set of points $(x_1, \ldots , x_{n_1-1}) \in \Real^{n_1-1}$, which
after substituting their coordinates instead of free variables in
a formula give a true statement, is called {\em semialgebraic}.

\subsection{Tarski-Seidenberg Theorem}
The most important fact about real formulae and semialgebraic sets is
the following
\medskip

\begin{theorem}[Tarski-Seidenberg]
Any semialgebraic set can be described as a union of sets, each of which is a set of all
solutions of a system of polynomial inequalities.
\end{theorem}

In other words, a set defined by a formula with quantifiers can be
defined by a formula without quantifiers.
The process of passing from one representation to another is called {\em quantifier elimination}.

\begin{remark} All the above definitions can be modified to
suit the field $\mathbb C$ of complex numbers rather than the reals.
Inequalities $f>0$ should then be replaced by $f \neq 0$.
Subsets of ${\mathbb C}^n$ defined by formulae with quantifiers are called ``constructible''.
Analogy of the Tarski-Seidenberg Theorem is also true.
\end{remark}

\subsection{Formal input}

The general problem we are concerned with in this course is quantifier elimination.
The input for this problem is of the kind
$$(Q_1 X^{(1)}) \cdots (Q_q X^{(q)}) P(Z, X^{(1)}, \ldots , X^{(q)}).$$

Here:
\begin{itemize}

\item
$Q_1, \ldots , Q_q$ are quantifiers $\exists$ or $\forall$;

\item
$X^{(i)}=(X_{1}^{(i)}, \ldots , X_{n_i}^{(i)})$ --- bound
variables;

\item $Z=(Z_1, \ldots , Z_{n_0})$ --- free variables;

\item $P$ --- Boolean combination of $k$ ``atomic'' formulae
of the kind $f>0$ or $f=0$, and for every $f$ the degree $\deg (f)<d$.
\end{itemize}

Set $n= n_0+n_1+ \cdots + n_q$.

\section{Complexity}

\subsection{Complexity for systems of inequalities}

Model of computation: Turing machine (or RAM --- random access machine).
\medskip

Consider the system of inequalities
\begin{equation}\label{eq:ineq}
f_1 \ge 0, \ldots ,  f_k \ge 0
\end{equation}
with $f_i \in {\mathbb Z}[X_1, \ldots X_n]$ (as usual, ${\mathbb Z}$ stands for
the ring of all integers), $\deg (f_i) <d$, maxima of bit lengths of
coefficients of $f_i$, $l(f_i) <M$.

We consider first the problem of deciding consistency of system (\ref{eq:ineq}).

There are some deep reasons to look for an upper complexity bounds of the kind
$$(kd)^EM^{O(1)}.$$

The first algorithm for the problem was proposed by A.Tarski in
late 40s \cite{T}.
In his bound $E$ can't be estimated from above by any tower of exponents.
Such bounds are called ``non-elementary'' and indicate that the corresponding
algorithm is absolutely impractical.

In mid-1970s G. Collins \cite{C} and, independently, H.R. W\"uthrich \cite{W} published an algorithm, called ``Cylindrical
Algebraic Decomposition'', for which $E=2^{O(n)}$.
Note that the Gr\"obner basis algorithm for the similar problem over
complex numbers has essentially the same complexity.

Our aim is to construct an algorithm with $E=O(n)$.
Thus, the complexity bound we are going to get is
$(kd)^{O(n)}M^{O(1)}$.

\subsection{Complexity for quantifier elimination}

For quantifier elimination problem the upper complexity bound has
the same form as for deciding consistency of systems:
$$(kd)^EM^{O(1)}.$$

For Cylindrical Algebraic Decomposition (CAD) the value of $E$
is again equal to $2^{O(n)}$.

Further in this course we are going to describe, with some detail,
an algorithm with
$$E= \prod_{0 \le i \le q} O(n_i) < (cn)^q$$
for a certain constant $c$.

\subsection{``Real numbers'' model}

Along with the Turing machine (bit) model of computation, one can consider
another model in which the elementary step is an arithmetic
operation or comparison operation over real numbers.
This model is more abstract in the sense that no constructive representation
of real numbers is assumed, the latter are viewed merely as symbols.
Real numbers model, therefore, can't serve as a foundation of algorithm's theory
in classical sense, however a formal theory of computation based on this
model exists (Blum-Shub-Smale theory).

Complexity results for real numbers model are usually stronger than similar
results for bit model.
One reason is that computations in bit model can take advantage of a clear
connection between the size of the input and metric characteristics of
definable semialgebraic sets, which allows sometimes to use rapidly
converging methods of numerical analysis.
Another technical advantage of bit model is a possibility to use
effective polynomial factorization as a subroutine which is not always
possible for polynomials with abstract coefficients.

On the weak side, real number model algorithm, being applied to polynomials
with essentially constructive (e.g., integer) coefficients, can't
control the growth of intermediate data.
Thus, it is theoretically possible that the bit complexity of an algorithm is
significantly worse than it's real number complexity.
In practice, however this almost never happens.

Tarski's algorithm works in real number model as well as for Turing machines.
Collins' CAD in original version can't handle symbolic coefficients, but
probably could be redesigned for real number model.

In this lecture course we follow the approach of James Renegar.
Our algorithms will work equally well in both models.

The complexity of algorithms in real number model will be of the
form $(kd)^E$ with
$E= O(n)$ for deciding consistency problem and
$E= \prod_{1 \le i \le q} O(n_i)$ for quantifier elimination problem.

\subsection{Lower bounds}

Observe that the number of the monomials (terms of the kind
$a_{i_1, \ldots , i_n}X_{1}^{i_1} \cdots X_{n}^{i_n}$) in a
polynomial in $n$ variables of the degree $d$ can be equal to
$$\binom{ d+n} {n } = d^{O(n)},$$
so the complexity $(kd)^{O(n)}$ for deciding consistency
is quite close to the ``best possible'' (for {\em dense} representation of polynomials, i.e.,
if we agree to write zero coefficients).
The best possible would be $kd^{O(n)}$, but the existence
of an algorithm with such complexity is a known open problem.

For quantifier elimination problem there is a result of Davenport
and Heintz \cite{DH} proving that polynomials occuring in quantifier-free
formula {\it after} quantifier elimination can have the degrees
not less than $d^{n^q}$.
It follows that Renegar's algorithm for quantifier elimination
is ``close to the best possible''.

\section{Univariate case: Sturm's theorem}

\subsection{Computational problems for univariate polynomials}

The idea behind almost all methods for solving the problems under
consideration, is to reduce the problems to the univariate case.
All straightforward methods, like CAD, do that, our effective
algorithms will use the same trick.

Strangely enough, we shall not consider the problem that seems the
most natural here: of {\em finding} solutions of univariate equations or inequalities.
This may mean producing a rational approximation with a given
accuracy to a root.
This problem belongs, in fact, to numerical analysis in the classical sense
and is rather far from the methods of this course.
However, it's worth mentioning that such an approximation indeed
can be produced in time polynomial in the length of the
description of input polynomials and accuracy bound.

So, if we don't ``compute'' roots (in the above sense) what is left algorithmically?
Over ${\mathbb C}$ (complex numbers) there is really not much to do
since, according to Fundamental Theorem of Algebra, any polynomial
$f$ has exactly $\deg_X(f)$ roots, and if we want to ``define'' a root
by an irreducible over rational numbers polynomial, we simply have
to factorize $f$.
Also, if we have a system of equations $f_1= \cdots =f_k=0$
we can decide its consistency by computing GCD (greatest common
divisor) over rational numbers of $f_1, \ldots , f_k$.

In real case, already deciding consistency of $f=0$ is a nontrivial task.

\subsection{Sturm's Theorem: formulation}

Let $f \in \Real[X]$, an interval $[a,b] \subset \Real$, the degree
$\deg (f) < d$.

Let $f'$ be the derivative of $f$.

Set
\medskip

\noindent $h_1=f$;
\medskip

\noindent $h_2=f'$;
\medskip

\noindent $h_{i+1}$ --- the reminder of division of $h_{i-1}$ by $h_i$,
taken with the opposite sign: $h_{i-1}=q_ih_i- h_{i+1}$,
$\deg (h_{i+1}) < \deg (h_i)$.
Thus, $h_{i+1}=q_ih_i- h_{i-1}$.
\medskip

Continue the process of producing polynomials $h_i$ while the division
is possible (i.e., until the reminder becomes a constant).
If the last (constant) reminder $h_t \neq 0$ (in this case $f$ and $f'$
are relatively prime), then divide all $h_i$'s by $h_t$.
The resulting sequence of polynomials $f_1=f/h_t, f_2=f'/h_t, f_3=h_3/h_t,
\ldots, f_{t-1}=h_{t-1}/h_t, f_t=1$
is called {\it Sturm's sequence}.
If $h_t =0$, then $f_1=f/h_{t-1}, f_2=f'/h_{t-1}, \ldots ,
f_{t-2}=h_{t-2}/h_{t-1},1$ is the {\it Sturm sequence}.

For any $x \in \Real$ let $V(x)$ be the number of {\it sign changes}
in the sequence of values of the Sturm's sequence: $f_1(x), f_2(x), \ldots$
Formally, $V(x)$ is the number of pairs $(i,j)$ of indices in Sturm's
sequence such that $i<j$, $f_i(x)f_j(x)<0$ and $f_k(x)=0$ if $i<k<j$.

\begin{theorem}[Sturm's Theorem] The number of distinct real roots
of $f$ in $[a,b]$ is $V(a) - V(b)$.
\end{theorem}

\begin{proof}
Can be easily found in literature.
\end{proof}

\subsection{Sturm's Theorem: discussion}

Sturm's Theorem gives and algorithm for deciding whether $f$ has a real root.
Indeed, from the formulation follows how to decide this for an interval $[a,b]$.
It's well known that {\it all} real roots of a univariate polynomial
$a_dX^d + \cdots + a_1X + a_0$ belong to the interval $[-R,R]$ with
$$R = 1 + \max_{0 \le i \le d-1} | a_i/a_d |.$$
So, it is enough to compute $V(-R)-V(R)$.
If this difference is $>0$ then there is a real root.

A simpler way is to pass from intervals to whole $\Real$ is to notice
that for any polynomial $f= a_dX^d + \cdots + a_1X + a_0$ the sign of
$f(\infty)$ (i.e., sign of $f$ at any sufficiently large point
$x \in \Real$) coincides with the sign of $a_d$ (or, equivalently, with the
sign of $a_dX^d$ at 1), while the sign of $f(-\infty)$ coincides with the
sign of $a_dX^d$ at $-1$.

Observe also that Sturm's Theorem provides conditions in terms of polynomial
equations and inequalities on coefficients of $f$, which are true if and only if
$f$ has a root in $\Real$.

\section{Univariate case: Ben-Or -- Kozen -- Reif algorithm}

\subsection{Towards generalization of Sturm's Theorem}

Our nearest goal is to generalize Sturm's Theorem to the case of several
univariate polynomials.

Let $f_1, \ldots , f_k \in \Real[X]$ with degrees $\deg (f_i)< d$ for all $i,\> 1 \le i \le k$.

\begin{definition}\label{def:css}
Consistent sign assignment for
$f_1, \ldots , f_k$ is a string
$$\sigma = ( \sigma_1, \ldots , \sigma_k ),$$
where $\sigma_i \in \{ >,<,= \}$, such that the system
$f_1 \sigma_1 0, \ldots , f_k \sigma_k 0$ has a root in $\Real$.
\end{definition}

We want to generalize Sturm's Theorem and construct an algorithm for listing all consistent sign assignments.
Moreover, we want to get polynomial conditions on coefficients,
guaranteeing the existence of a common root.

The following algorithm belongs to M. Ben-Or, D. Kozen, and J. Reif \cite{BKR}.

\subsection{Preparation of the input}

First we ``prepare'' the input family so that it would be more simple.
\medskip

\noindent (1)\quad Make all $f_i$'s {\it squarefree} and {\it relatively prime}.
It is sufficient to compute greatest common divisor (GCD) $g_i$
of $f_i$ and its derivative $f'_i$ for all $1 \le i \le k$ (with the help of
Euclidean algorithm), and then to divide each $g_i$ by GCD $g$ of
the family $g_1, \ldots , g_k$.
Now we can reconstruct all consistent sign assignments of $f_1, \ldots , f_k$
from all consistent sign assignments of $g, g_1, \ldots , g_k$.

Throughout this portion of notes we assume that
all $f_i$'s are {\it squarefree} and {\it relatively prime}.
It follows that all the sets of all roots of polynomials $f_i$ are disjoint,
thus {\it any} consistent sign assignment for the family
has {\it at most one} ``='' sign.
\medskip

\noindent (2)\quad Add a new member $f$ in the family $f_1, \ldots , f_k$,
so that any consistent sign assignment for $f, f_1 , \ldots , f_k$ has
{\it exactly} one ``='' sign, and
each consistent sign assignment for $f_1, \ldots , f_k$ can be obtained
from a consistent sign assignment for $f, f_1, \ldots , f_k$.
Let the product $f_1 \cdots f_k = a_{kd}X^{kd}+ \cdots + a_0$.
Define
$$R= 1+ \max_{0 \le j < kd} \left\{ \frac{a_j}{a_{kd}} \right\}.$$
Then we can take
$$f= \frac{ d (f_1 \cdots f_k)}{d X } (X-R) (X + R).$$
Now the problem can be reformulated as follows:
list all consistent assignments of signs $<$ and $>$ for
$f_1, \ldots , f_k$ at roots of $f$.

\subsection{Cases $k=0$ and $k=1$}

For $k=0$ use Sturm's sequence obtained by Euclidean algorithm.

Denote $V(- \infty)-V ( \infty ) = S(f,f')$ (recall that we start
with Euclidean division of $f$ by $f'$).

In case $k=1$ we have a family of two polynomials $f, f_1$.

Denote:
$$ C_1 = \{ f=0\> \&\> f_1 >0 \};$$
$$ \overline {C_1} = \{ f=0\> \&\> f_1<0 \}.$$
For a finite set $A$, let $|A|$ be the number of elements in $A$.
Obviously, $S(f,f')= |C_1| + |\overline {C_1}|$.

The following trick is very important for generalization of Sturm's
Theorem.
Construct Sturm's sequence of remainders starting with dividing
$f$ by $f'f_1$ (rather than simply by $f'$ as in original
Sturm's sequence).
Denote the difference of numbers of sign changes at $- \infty$ and
at $\infty$ by $S(f,f'f_1)$.
\bigskip

\begin{lemma}\label{le:difference}
$S(f, f'f_1)= |C_1|-|\overline {C_1}|$.
\end{lemma}

\begin{proof}
Trivial modification of a proof of original
Sturm's Theorem.
We don't consider the proof in this course.
\end{proof}

Original Sturm and this Lemma imply that the following system of linear
equations with the unknowns $|C_1|, |\overline {C_1}|$ is true:
$$
\left( \begin{array}{cc}
1 & 1\\
1 & -1
\end{array} \right)
\left( \begin{array}{c}
|C_1|\\
| \overline {C_1} |
\end{array} \right)=
\left( \begin{array}{c}
S(f, f')\\
S(f, f'f_1)
\end{array} \right).
$$

Denote the $(2 \times 2)$-matrix of this system by $A_1$.

The value $|C_1|$ determines whether  $\{ f=0\> \&\> f_1>0 \}$ is consistent.
The value $| \overline {C_1} |$ determines whether $ \{ f=0\> \&\> f_1<0 \}$
is consistent.
Hence, after the linear system is solved all consistent sign assignments
will be listed.

\subsection{Case $k=2$}\label{ss:k=2}

In this case we come to the following system of linear equations:
$$\left( \begin{array}{cccc}
1 & 1 & 1 & 1\\
1 & -1 & 1 & -1 \\
1 & 1 & -1 & -1\\
1 & -1 & -1 & 1
\end{array} \right)
\left( \begin{array}{c}
|C_1 \cap C_2|\\
| \overline {C_1} \cap C_2 |\\
|C_1 \cap \overline {C_2} |\\
| \overline {C_1} \cap \overline {C_2} |
\end{array} \right)
=
\left( \begin{array}{c}
S(f, f')\\
S(f, f'f_1)\\
S(f, f'f_2)\\
S(f, f'f_1f_2)
\end{array}
\right),
$$
where
$$C_2= \{ f=0\> \&\> f_2>0 \},\quad
\overline {C_2}= \{ f=0\> \&\> f_2<0 \}.$$

Denote the matrix of this system by $A_2$.

Observe, that if this system of linear equations is true, then all
consistent sign assignments for $f, f_1, f_2$ will be found.

Checking that the system is indeed true, is straightforward:
\medskip

\noindent (1)\quad $|C_1 \cap C_2| + |C_1 \cap \overline {C_2}|$
is the number of all roots of $f$ with $f_1>0$, while
$| \overline {C_1} \cap C_2| + | \overline {C_1} \cap \overline {C_2}|$
is the number of all roots of $f$ with $f_1<0$, so the left-hand part
of the first equation is the whole number of roots of $f$;
\medskip

\noindent (2)\quad $|C_1 \cap C_2| + |C_1 \cap \overline {C_2}|$
is the number of all roots of $f$ with $f_1>0$, while
$| \overline {C_1} \cap C_2| + | \overline {C_1} \cap \overline {C_2}|$
is the number of all roots of $f$ with $f_1<0$, so the left-hand part
of the second equation equals to the difference between these two
quantities, which is $S(f, f'f_1)$ by the Lemma~\ref{le:difference};
\medskip

\noindent (3)\quad similar to (2);
\medskip

\noindent (4)\quad $|C_1 \cap C_2| + | \overline {C_1} \cap \overline {C_2}|$
is the number of all roots of $f$ with $f_1$ and $f_2$ having the same sign
(i.e., $f_1f_2>0$), while
$| \overline {C_1} \cap C_2| + |C_1 \cap \overline {C_2}|$
is the number of all roots of $f$ with $f_1$ and $f_2$ having distinct signs
(i.e., $f_1f_2<0$), so the left-hand part
of the last equation is $S(f, f'f_1f_2)$ by the Lemma~\ref{le:difference}
(applied to $f_1f_2$ instead of $f_1$).

\subsection{Tensor product of matrices}

Observe that $A_2= A_1 \otimes A_1$, where $\otimes$ is the operation of
{\em tensor product}.

\begin{definition}
Let $A= \| a_{ij}\|_{1 \le i \le m;
1 \le j \le n}$ and $B= \| b_{ij} \|_{1 \le i \le m'; 1 \le j \le n'}$
be two matrices.
Then the tensor product of $A$ by $B$ is the $(mm' \times nn')$-matrix
$$
A \otimes B=
\left( \begin{array}{ccc}
a_{11}B & \cdots & a_{1n}B\\
\cdots & \cdots & \cdots\\
a_{m1}B & \cdots & a_{mn}B
\end{array} \right)
,$$
where
$$a_{ij}B=
\left( \begin{array}{ccc}
a_{ij}b_{11} & \cdots & a_{ij}b_{1n'}\\
\cdots & \cdots & \cdots\\
a_{ij}b_{m'1} & \cdots & a_{ij}b_{m'n'}
\end{array} \right)
.$$
\end{definition}

\subsection{Case of arbitrary $k$}

Let us consider now the case of a family $f, f_1, \ldots , f_k$ for arbitrary $k$.
Define the matrix $A_k= A_1 \otimes A_1 \otimes \cdots \otimes A_1$ ($k$ times).
Let $C$ be the vector of all possible elements of the form
$|D_1 \cap \cdots \cap D_k|$ where $D_i \in \{ C_i, \overline {C_i} \},\>
1 \le i \le k$; and $S$ be the vector of all possible elements of the form
$S(f, f'f_{j_1} \cdots f_{j_t})$ for all $(j_1, \ldots , j_t) \subset
\{1, \ldots , k \}$.
Herewith, the component of $C$ having ``bared'' sets $\overline {C_{j_1}}, \ldots ,
\overline {C_{j_t}}$ corresponds to the component
$S(f, f'f_{j_1} \cdots f_{j_t})$ of $S$.

\begin{lemma}
$A_kC=S$
\end{lemma}

\noindent {\bf Proof.}\quad Streightforward observation (similar to particular
case of $k=2$ in Section~\ref{ss:k=2}).
\medskip

Square matrix $A_k$ is nonsingular.
Its unique solution describes all consistent sign assignments for
$f, f_1, \ldots , f_k$.
Solving the system by any of standard methods (e.g., Gaussian elimination)
we list all consistent sign assignments.
Let us estimate the complexity of this algorithm.

From definition of tensor product it follows that the order of $A_k$ is
$2^k \times 2^k$.
Vectors $C$ and $S$ have obviously $2^k$ components each.
Each component of the vector $S$ can be computed
(in real numbers model) by Euclidean division
in time polynomial in $kd$, and we need to find $2^k$ components.
In bit model, supposing that the bit sizes of integer coefficients of
polynomials $f, f_1, \ldots , f_k$ are less than $M$, vector $S$
can be computed with complexity polynomial in $M2^k d $.
The complexity of Gaussian algorithm is polynomial in the order of the matrix,
i.e., in $2^k$ (in real numbers model) and $M 2^k$ (in bit model).
Thus, the total running time of finding all consistent sign assignments
via system $A_kC=S$ is polynomial in $2^kd$ (real numbers) or in $M2^kd$ (bit).

We want, however, to obtain an algorithm whose running time is polynomial
in $k,d$ (real numbers) and in $M, k, d$ (bit).

\subsection{Divide-and-conquer algorithm}

A key observation to get a better algorithm is that since $\deg (f_i) < d$
and $\deg (f) < kd+2$, the number of roots of $f$ is less than $kd+2$, and
therefore, the total number of all consistent sign assignments is less than
$kd+2$.

It follows that among the components of the solution of the system $A_kC=S$
there are less than $kd+2$ nonzero ones.
Thus, $A_kC=S$ is equivalent to a much smaller (for large $k$)
$\bigl( (kd+2) \times (kd+2) \bigr)$-linear system.
Now the problem is how to construct that smaller system effectively.

We shall apply the usual divide-and-conquer recursion on $k$
(supposing for simplicity of notations that $k$ is even).

Suppose that we can do all the computation in the case of $k/2$.
Namely, assume that we already have the solutions for two problems
for two families of polynomials:
$$f, f_1, \ldots , f_{k/2}\quad {\rm and}\quad f, f_{k/2+1}, \ldots , f_k.$$
This means, we have constructed:
\medskip

\noindent (1)\quad two square nonsingular matrices $A$ and $A'$ with elements
from $\{ -1, 1 \}$, of orders $m \times m$ and $n \times n$ respectively,
$m,n < kd+2$;
\medskip

\noindent (2)\quad two vectors $S$ and $S'$ of sizes $m$ and $n$ respectively
with elements of the forms
$$S(f, f' \prod I),\> I \subset \{ f_1, \ldots , f_{k/2} \}$$
and
$$S(f, f' \prod J),\> J \subset \{ f_{k/2+1}, \ldots , f_k \}$$
respectively, such that the systems $AC=S,\> A'C'=S'$ have solutions
$C,C'$ respectively with {\it nonzero coordinates only} of the forms
$$\Big\arrowvert \bigcap_{q \in I}C_q \cap \bigcap_{q \in
\{ f_1, \ldots , f_{k/2} \}
\setminus I } \overline {C_q} \Big\arrowvert,\> I \subset
\{ f_1, \ldots , f_{k/2} \},$$
and
$$\Big\arrowvert \bigcap_{q \in J}C_q \cap \bigcap_{q \in
\{ f_{k/2+1}, \ldots , f_{k} \}
\setminus J } \overline {C_q} \Big\arrowvert,\> J \subset
\{ f_{k/2+1}, \ldots , f_{k} \}$$
respectively;
\medskip

\noindent (3)\quad solutions $C$ and $C'$ which correspond to all consistent
sign assignments for $f_1, \ldots , f_{k/2}$ and $f_{k/2+1}, \ldots , f_k$
respectively.
\medskip

Now let us construct a new linear system.
As a matrix take $A \otimes A'$, which is a square nonsingular
$(mn \times mn)$-matrix.

Two vectors $S,S'$ are combined in a new column vector $SS'$ consisting of
$m$ column vectors placed end-to-end, each of length $n$.

The $(i,j)$th component of $SS'$ is
$$S \bigl( f, f'\prod (I \triangle J) \bigr),$$
where $S(f, f' \prod I)$ is the $i$th component of $S$,
$S(f, f' \prod J)$ is the $j$th component of $S'$, and
$\triangle$ denotes the {\em symmetric difference} of two sets of polynomials:
$$I \triangle J = (I \cup J) \setminus (I \cap J).$$

Finally, two vectors $C$ and $C'$ are combined into a new column vector
$CC'$.

Components of $CC'$ are indexed by pairs $(i,j)$.
The $(i,j)$th component of $CC'$ is $|E \cap E'|$ where $|E|$ is the $i$th
component of $C$ and $|E'|$ is $j$th component of $C'$.

\begin{lemma}\label{le:tensor}
$(A \otimes A')(CC')=SS'$.
\end{lemma}

We don't consider a proof of Lemma~\ref{le:tensor} in the course.
\medskip

Observe that $CC'$ describes all consistent sign assignments for
$f, f_1, \ldots , f_k$.
However, unlike vectors $C$ and $C'$, the vector $CC'$ {\em may} have
zero components.
As was already noted, the new system has the order $mn \times mn$.
It may happen that $mn > kd+2$.
In this case $CC'$ indeed contains zero components, not less than
$mn - (kd+2)$.

By solving system, we find $CC'$, and after that {\em reduce} the system
$(A \otimes A')(CC')=SS'$.
Namely, we delete all zero components of $CC'$ and the corresponding columns
of $A \otimes A'$ to obtain a rectangular system of order at most $mn \times (kd+2)$.
After that, choose a basis among the rows of the matrix and form a new
$\bigl( (kd+2) \times (kd+2) \bigr)$-matrix.
Delete all components of $SS'$ {\it not} corresponding to basis rows.

The resulting system is the output of the recursive step of the algorithm.

\subsection{Complexity of Ben-Or--Kozen--Reif algorithm}

It's easy to see that the algorithm has the running time polynomial in $k,d$ (and $M$, in bit model).
Indeed, denoting the complexity function by $T(k)$, we get a functional
inequality
$$T(k) \le 2T(k/2) + p(kd),$$
for a certain polynomial $p$, whose solution $T(k) \le O \bigl( p(kd) \bigr)$.

\section{Systems of polynomial inequalities}\label{sec:systems}

\subsection{Deciding consistency of systems of polynomial inequalities}

We address the basic problem of deciding consistency of systems of inequalities
\begin{equation}\label{eq:nonstrict-strict}
f_1 \ge 0, \ldots , f_{k_1} \ge 0, f_{k_1+1} > 0, \ldots , f_k >0.
\end{equation}
Here polynomials $f_i \in {\mathbb Z}[X_1, \ldots , X_n]$ (for bit model)
or $f_i \in \Real[X_1, \ldots , X_n]$ (for real numbers model).
A version of the algorithm presented below belongs to James Renegar.

It is convenient to deal with systems of nonstrict inequalities only.
Introduce a new variable $Y$ and replace strict inequalities in (\ref{eq:nonstrict-strict}) by
$$Yf_{k_1+1} \ge 1, \ldots , Yf_k \ge 1, Y \ge 0.$$
Obviously, (\ref{eq:nonstrict-strict}) is consistent if and only if the new system is consistent.
Changing the notations, we shall consider, in what follows, the problem of consistency
of a system of nonstrict inequalities
\begin{equation}\label{eq:nonstrict}
f_1 \ge 0, \ldots , f_k \ge 0.
\end{equation}

Because it is easy to check whether or not the origin $(0, \ldots , 0)$ is a solution
of (\ref{eq:nonstrict}), we shall assume that the origin does not satisfy (\ref{eq:nonstrict}).

Finally, we suppose that the set $\{ (\ref{eq:nonstrict}) \}$ is bounded, if nonempty.
The nonbounded case can be reduced to a bounded by intersecting $\{ (\ref{eq:nonstrict}) \}$
with a ball in $\Real^n$ of a sufficiently large radius, we are not
considering the details of this construction in the course.

\subsection{Towards reduction to a system of equations}\label{sec:epsilon}

Let the degrees $\deg (f_i) < d$ for all $1 \le i \le k$, and $D$
be the minimal even integer not less than $kd + 1$.
Introduce a new variable $\varepsilon$, and consider the polynomial
$$g(\varepsilon)= \prod_{1 \le i \le k} (f_i + \varepsilon)
-\varepsilon^{k+1} \sum_{1 \le j \le n}X_{j}^{D}. $$
\bigskip

\begin{lemma}\label{le:2to3}
If (\ref{eq:nonstrict}) is consistent, then there is a solution of
(\ref{eq:nonstrict}) which is a limit of a root of
\begin{equation}\label{eq:system-with-g}
\frac{\partial g(\varepsilon)}{\partial X_1} = \cdots =
\frac{\partial g(\varepsilon)}{\partial X_n} =0
\end{equation}
as $\varepsilon \longrightarrow 0$.
\end{lemma}

\begin{proof}
Let $x \in \{ (\ref{eq:nonstrict}) \}$.
Then for all sufficiently small positive values of $\varepsilon$
$$\varepsilon^{k+1} \sum_j  x_{j}^{D} < \prod_i (f_i(x) +\varepsilon),$$
hence $g(\varepsilon)(x) >0$.
Fix for a time being one of these values of $\varepsilon$.

Let $C$ be any connected component of $\{ g(\varepsilon)>0 \}$ such that $x \in C$.
Note that at all points $v \in C$ the sign of $f_i(v)+ \eps$ is the same for each $i \in \{ 1, \ldots ,k \}$,
hence $f_i(v)+ \eps >0$.

Observe that $C$ is bounded.
Indeed, suppose $C$ is {\em not bounded}.
Then, there exists a point $y \in C \setminus \{(5.2) \}$ at an arbitrarily large distance
from $\{ (5.2) \}$.
In particular, for such $y$ there is at least one $i \in \{ 1, \ldots ,k \}$, with $f_i(y) <0$,
having a large distance ${\rm dist} (y, \{ f_i(y)=0 \})$.
It follows from basic facts of analytic geometry ({\L}ojasiewicz inequality), that for
${\rm dist} (y, \{ f_i(y)=0 \})$ sufficiently larger than $\eps$, we have $| f_i(y)|> \eps$, i.e.,
$f_i(y)+ \eps <0$.
Since $C$ is connected, there is a connected (semialgebraic) curve
$\Gamma \subset C$ containing $x$ and $y$.
Because $f_i(x)+\eps > 0$ and $f_i(y)+ \eps <0$, there exists a point
$z \in \Gamma$ with $f_i(z)+ \eps=0$.
We get $g(\varepsilon)(z)= - \varepsilon^{k+1} \sum_jz_{j}^{D} \le 0$, which contradicts
$z \in \{ g(\varepsilon)>0 \}$.

At every point $w$ in the {\it boundary} of $C$ polynomial $g(\varepsilon)(w)=0$.
It follows that $g(\varepsilon)$ attains a {\it local maximum}
at a point $v \in C$.
It is well known that any such point $v$ is a root of the system of equations
$$\frac{\partial g(\varepsilon)}{\partial X_1}(v) = \cdots =
\frac{\partial g(\varepsilon)}{\partial X_n}(v) =0.$$
Also, since $v \in C$, we have
$$f_1(v)+\varepsilon > 0, \ldots , f_k(v)+\varepsilon > 0.$$
Passing to limit in the latter system of inequalities, we get
$$f_1(v') \ge 0, \ldots , f_k(v') \ge 0,$$
where $v \longrightarrow v'$ as $\varepsilon \longrightarrow 0$.

The lemma is proved.
\end{proof}

\subsection{Homogeneous polynomials and projective spaces}

Any polynomial is an expression of the form
$$\sum_{(i_1, \ldots , i_n)} a_{i_1, \ldots , i_n}
X_{1}^{i_1} \cdots X_{n}^{i_n}.$$

A polynomial is called {\it homogeneous polynomial} or {\it form}
if $i_1 + \cdots + i_n$ is a non-zero constant for all $(i_1, \ldots , i_n)$
with $a_{i_1, \ldots , i_n} \neq 0$.
\bigskip

\noindent {\bf Examples.}
\medskip

\noindent (i)\quad Homogeneous polynomial: $X^2Y^3 + X^5 + XY^4$;
\medskip

\noindent (ii)\quad Homogeneous polynomial: $X^2 - Y^2$;
\medskip

\noindent (iii)\quad Nonhomogeneous polynomial: $X^2-1$.
\medskip

Observe that zero is always a root of a homogeneous polynomial.
Also, if $x=(x_1, \ldots , x_n) \in {\mathbb C}^n$ is a root of a homogeneous polynomial
$h \in {\mathbb C}[X_1, \ldots , X_n]$,
then for any $a \in {\mathbb C}$ the point $ax=(ax_1, \ldots , ax_n)$ is again
a root of $h$.
For a fixed $x \neq 0$ the set of all points of the form $ax,\> a \in {\mathbb C}$ is a straight
line passing through the origin.
So, if $x \neq 0$ is a root of $h$, then the line passing through $x$ and the
origin consists of roots of $h$.
Thus, we can say that the {\it line itself is a root} of $h$.

This motivates the idea to consider (in homogeneous case)
a space of {\it straight lines}, passing through
the origin, instead of the usual space of {\it points}.
It is called an $(n-1)$-{\it dimensional projective space}
and is denoted by ${\mathbb P}^{n-1}({\mathbb C})$.
An element of ${\mathbb P}^{n-1}({\mathbb C})$, which is a line passing through
$x=(x_1, x_2, \ldots , x_n) \neq (0, \ldots , 0)$, is denoted by
$(x_1: x_2: \cdots : x_n)$ (i.e., we use colon instead of comma).
Of course, for any $0 \neq a \in {\mathbb C}$ elements of ${\mathbb P}^{n-1}({\mathbb C})$:
$$(x_1: x_2: \cdots : x_n)\quad {\rm and}\quad (ax_1: ax_2: \cdots : ax_n)$$
coincide.
The usual space of points ${\mathbb C}^n$ is called {\it affine} if there is
a possibility of a confusion.

Observe that our construction is leading to the projective space
of the dimension $n-1$, which
is quite natural as can be seen from the example of
${\mathbb P}^1({\mathbb C})$.
The latter is the space of all straight lines on the plane $P$
of (complex) coordinates $X, Y$, passing through the origin.
All elements of the projective space ${\mathbb P}^1({\mathbb C})$,
except one, correspond bijectively to all points of the line $\{ X=1 \}$.
Namely, if $(x:y) \in {\mathbb P}^1({\mathbb C})$ and $x \neq 0$,
then $(x:y)= \bigl( 1:(y/x) \bigr) \in {\mathbb P}^1({\mathbb C})$,
and {\it the point} of $P$ $\bigl( 1, (y/x) \bigr) \in \{ X=1 \}$.
The only element, for which this correspondence fails, is
$(0:y)$ (for all $y \in {\mathbb C}$ this is the same element,
called {\it point at infinity}).
Thus, we have a parametrization of ${\mathbb P}^1({\mathbb C})$
minus point at infinity by points of the line.
This justifiers the definition of the dimension of ${\mathbb P}^1({\mathbb C})$
as ``one''.

An arbitrary polynomial $h \in {\mathbb C}[X_1, \ldots , X_n]$ can be naturally
homogenized.
Namely, let $\deg_{X_1, \ldots, X_n}(h) =d$.
Introduce a new variable $X_0$ and consider the homogeneous polynomial
${\rm hom}(h)=X_{0}^{d}h(X_1/X_0, \ldots , X_n/X_0)$.
If $(x_1, \ldots , x_n)$ is a root of $h$, then $(1:x_1: \cdots :x_n)$
is a root of ${\rm hom}(h)$.
On the other hand, if $(y_0:y_1: \cdots :y_n)$ is a root of ${\rm hom}(h)$
and $y_0 \neq 0$, then $(y_1/y_0, \ldots , y_n/y_0)$ is a root of $h$.
So, there is a bijective correspondence between roots of $h$ and
roots of of ${\rm hom}(h)$ with $y_0 \neq 0$.
However, ${\rm hom}(h)$ can have a root of the form $(0:z_1: \cdots :z_n)$
(then, of course, at least one of $z_1, \ldots , z_n$ is different from 0).
This root does not correspond to any root of $h$ (in the sense described
above), it's called ``root at infinity''.
\bigskip

\noindent {\bf Example.}
\medskip

\noindent $h=X_1-X_2+1;$

\noindent ${\rm hom}(h)=X_1-X_2+X_0.$

\noindent $(0:1:1)=(0:a:a)$ for any $0 \neq a \in {\bf C}$,
is the only root at infinity of ${\rm hom}(h)$.
\bigskip

The described above relation between the roots of a polynomial and it's
homogenization can be extended to roots of {\it systems of polynomial
equations}.
A system can have {\it no roots} (i.e., be inconsistent), while its homogenization can have roots.
\bigskip

\noindent {\bf Example.}
\medskip

\noindent $X_1-X_2=X_1-X_2+1=0;$

\noindent homogenization: $X_1-X_2=X_1-X_2+X_0=0.$

\noindent The first system is inconsistent, while the homogenized one
has the unique root at infinity $(0:1:1)$.
\bigskip

We conclude that the homogenized system has not less roots than the
original system.
This implies, for instance, that if the homogenized system has
{\it at most finite number of roots}, then the original system
has at most finite number of roots as well.
We shall use this remark further on.

The reason for considering homogeneous equations in projective
space instead of arbitrary ones in affine space, is that the projective
case is much richer in useful properties.

\subsection{$u$-resultant}

Consider a system of $n$ {\it homogeneous} polynomial equations
in variables $X_0, X_1, \ldots , X_n$:
\begin{equation}\label{eq:hom-eq}
h_1= \cdots =h_n=0.
\end{equation}

Let $\deg (h_i) < d$ for each $i \in \{ 1, \ldots ,n \}$.

For some deep algebraic reasons,
which we are not discussing in this course,
the system (\ref{eq:hom-eq}) is {\it always consistent} in ${\mathbb P}^n({\mathbb C})$.
It can have either finite or infinite number of roots.

Introduce new variables $U_0, U_1, \ldots , U_n$.
\bigskip

\begin{lemma}\label{le:u-res}
There exists a homogeneous polynomial
$R_{h_1, \ldots , h_n} \in {\mathbb C}[U_0, \ldots , U_n]$ called $u$-resultant,
such that $R_{h_1, \ldots , h_n} \equiv 0$ (is identically zero) if and
only if (\ref{eq:hom-eq}) has infinite number of roots.
Moreover, if $R_{h_1, \ldots , h_n} \not\equiv 0$
(and, thus (\ref{eq:hom-eq}) has finite number $r$ of roots), then it can be decomposed
into linear factors:
$$R_{h_1, \ldots , h_n}= \prod_{1 \le i \le r}
(\chi_{0}^{(i)}U_0 + \chi_{1}^{(i)}U_1 + \cdots + \chi_{n}^{(i)}U_n)$$
where $(\chi_{0}^{(i)}: \cdots :\chi_{n}^{(i)}),\> 1 \le i \le r$ are
all the roots of (\ref{eq:hom-eq}) in ${\mathbb P}^n({\mathbb C})$.
The degree of $R$ is $d^{O(n)}$.
\end{lemma}

We are not considering the full proof of this important lemma in this course.
At certain point our algorithm will need to {\em compute} the $u$-resultant.
At that point we shall give a sketch of the procedure together with some
hints on a proof of Lemma~\ref{le:u-res}.

\subsection{System (\ref{eq:system-with-g}) has a finite number of roots}

\begin{lemma}\label{le:finite-number}
For all sufficiently small positive values of
$\varepsilon$, if (\ref{eq:nonstrict}) is consistent, then (\ref{eq:system-with-g}) has a finite number of roots.
\end{lemma}

\begin{proof}
We already know (Lemma~\ref{le:2to3}) that if (\ref{eq:nonstrict})
is consistent, then (\ref{eq:system-with-g}) has at least one root.
We also know (end of Section~5.3) that it's sufficient to prove that the
homogenization of (\ref{eq:system-with-g}) has finite number of roots in the projective
space ${\mathbb P}^n({\mathbb C})$.

Let $G_i(\varepsilon)$ be the homogenization of $\partial g(\varepsilon) /
\partial X_i$.
We shall now prove that
\begin{equation}\label{eq:homogen}
G_1(\varepsilon)= \cdots = G_n(\varepsilon)=0
\end{equation}
has finite number of roots in ${\mathbb P}^n({\mathbb C})$ for all sufficiently small $\eps >0$.

Homogeneous polynomial $G_i(\varepsilon)$ is of the form:
$$G_i(\varepsilon)=F(\varepsilon)- \varepsilon^{k+1}DX_{i}^{D-1},$$
where $F(\varepsilon)$ is a homogeneous polynomial (w.r.t. variables
$X_0, \ldots , X_n$).

The degree of $F(\varepsilon)$ (w.r.t. $\varepsilon$) is $k$.
Thus, for any fixed values of $X_0, \ldots , X_n$, we have $F(\varepsilon)/ \varepsilon^{k+1} \longrightarrow 0$ as
$\varepsilon \longrightarrow \infty$.

Observe that for $\varepsilon \neq 0$ the system (\ref{eq:homogen}) has the same
set of roots as
\begin{equation}\label{eq:over-eps}
\frac{G_1(\varepsilon)}{\varepsilon^{k+1}}= \cdots = \frac{ G_n(\varepsilon)}
{\varepsilon^{k+1}}=0.
\end{equation}

The $u$-resultant of (\ref{eq:over-eps}) is either identically zero (as a polynomial in
$U_0, \ldots , U_n$) for {\it all} values of $\varepsilon$, or is
identically zero for {\it at most finite number} of values of
$\varepsilon$.

The first alternative does not occur, because for {\it large} positive values of
$\varepsilon$, the system (\ref{eq:over-eps}) tends to
\begin{equation}\label{eq:eps-large}
X_{1}^{D-1}= \cdots = X_{n}^{D-1}=0,
\end{equation}
so the $u$-resultant of (\ref{eq:over-eps}) tends to $u$-resultant of (\ref{eq:eps-large}).
But (\ref{eq:eps-large}) has exactly one root in ${\mathbb P}^n({\mathbb C})$
(namely, $(1:0: \cdots : 0)$), so $u$-resultant for (\ref{eq:eps-large}), and hence
for (\ref{eq:over-eps}) is not identically zero for large values of $\varepsilon$.

It follows that the $u$-resultant is not identically zero also for
all sufficiently {\it small} positive values of $\varepsilon$,
so for these values (\ref{eq:homogen}) has a finite number of roots.
\end{proof}

\subsection{General scheme of the algorithm}

Recall that, according to Lemma~\ref{le:2to3}, if (\ref{eq:nonstrict}) is consistent,
then there is a solution to (\ref{eq:nonstrict}) which is a limit of a root of (\ref{eq:system-with-g})
as $\varepsilon \longrightarrow 0$.

On the other hand, for all sufficiently small positive values of $\varepsilon$
(\ref{eq:system-with-g}) has a finite number of roots, due to Lemma~\ref{le:finite-number}.
Moreover, the homogenization (\ref{eq:homogen}) of (\ref{eq:system-with-g}) has a finite number of roots.
Now Lemma~\ref{le:u-res} implies that every root of (\ref{eq:homogen}) is a coefficient
vector in a divisor of $u$-resultant $R_{G_1(\varepsilon), \ldots ,
G_n(\varepsilon)}$.

If we were able to factorize $R_{G_1(\varepsilon), \ldots ,
G_n(\varepsilon)}$, over ${\mathbb C}$, then the process of deciding
the consistency of (\ref{eq:nonstrict}) could be roughly as follows:
construct the $u$-resultant for (\ref{eq:homogen}), factorize, then
for each coefficient vector $(\chi_{0}^{(i)}, \ldots , \chi_{n}^{(i)})$
of each factor of the $u$-resultant find its component-wise limit
$(\eta_{0}^{(i)}, \ldots , \eta_{n}^{(i)})$, and check whether the point
$$(\eta_{1}^{(i)}/\eta_{0}^{(i)}, \ldots , \eta_{n}^{(i)}/\eta_{0}^{(i)})$$
satisfies the system (\ref{eq:nonstrict}).
If yes, then (\ref{eq:nonstrict}) is consistent, else (\ref{eq:nonstrict}) is inconsistent.

Unfortunately, we are unable to factorize a polynomial in general case.
Even in special cases, when effective multivariate factorization
algorithms are known, they are very involved.

Therefore our scheme will be different.
After constructing the $u$-resultant, we shall find another polynomial
in $U_0, \ldots , U_n$, whose factors correspond to limits of factors
of the $u$-resultant.

Instead of factoring this new polynomial, we shall use the fact that
the set of its roots in ${\mathbb C}^{n+1}$ is a union of
a finite number of {\it hyperplanes}, passing through the origin.
The coefficient vectors defining these hyperplanes are exactly
the coefficient vectors of factors and can be found by computing
the gradients of the polynomial at the appropriate points.

\subsection{Passing to limit in the $u$-resultant}

Recall that $R_{G_1(\varepsilon), \ldots , G_n(\varepsilon)}$ is
a polynomial in the coefficients of homogeneous polynomials
$G_i(\varepsilon)$.
In particular, it is a polynomial in $\varepsilon$.

Represent $R_{G_1(\varepsilon), \ldots , G_n(\varepsilon)}$
in the form:
$$R_{G_1(\varepsilon), \ldots , G_n(\varepsilon)}=
R_m \varepsilon^m + \sum_{j>m} R_j \varepsilon^j,$$
where $R_m$ is not identically zero in $U_0, \ldots , U_n$,
i.e., $m$ is the {\it lowerst} degree of a term
in $\varepsilon$ of $R_{G_1(\varepsilon), \ldots ,
G_n(\varepsilon)}$.
\bigskip

\begin{lemma}\label{le:u-res-m}
$$R_m= \prod_{1 \le i \le r} (\eta_{0}^{(i)}U_0 + \cdots + \eta_{n}^{(i)}U_n),$$
where $(\eta_{0}^{(i)}, \ldots , \eta_{n}^{(i)})$ is the limit of
$(\chi_{0}^{(i)}, \ldots , \chi_{n}^{(i)})$ as $\varepsilon \longrightarrow 0$
(some, but not all, of the $\eta^{(i)}$-vectors can be zero).
\end{lemma}

\begin{proof}
Consider
$$\frac{ R_{G_1(\varepsilon), \ldots , G_n(\varepsilon)}}{ \varepsilon^m }=
R_m + \sum_{j>m} R_j \varepsilon^{j-m}.$$

As $\varepsilon$ tends to zero, the polynomial
$R_{G_1(\varepsilon), \ldots , G_n(\varepsilon)}/ \varepsilon^m$ tends
to the polynomial $R_m$.
Hence the roots (in the projective space ${\mathbb P}^n({\mathbb C})$
with projective coordinates $U_0, \ldots , U_n$) of
$R_{G_1(\varepsilon), \ldots , G_n(\varepsilon)}$ tend to roots of
$R_m$, as $\varepsilon \longrightarrow 0$.

Because $R_{G_1(\varepsilon), \ldots , G_n(\varepsilon)}$ factors
linearly, the set of all roots of this polynomial is a
collection of $r$  hyperplanes
$\left\{ \chi_{0}^{(i)}U_0 + \cdots + \chi_{n}^{(i)}U_n =0 \right\}$.
Each hyperplane is uniquely described by its normal vector
$(\chi_{0}^{(i)}, \ldots , \chi_{n}^{(i)})$.

If follows that the set of all roots of $R_m$ is a collection
of {\it at most} $r$ hyperplanes, each of which is uniquely described
by its normal vector
$$\bigl( \lim_{\varepsilon \longrightarrow 0} (\chi_{0}^{(i)}), \ldots ,
\lim_{\varepsilon \longrightarrow 0} (\chi_{n}^{(i)}) \bigr).$$
Denoting
$$\eta_{j}^{(i)}= \lim_{\varepsilon \longrightarrow 0} (\chi_{j}^{(i)}) $$
for $1 \le i \le r;\> 0 \le j \le n$, we get:
$$R_m= \prod_{1 \le i \le r} (\eta_{0}^{(i)}U_0 + \cdots + \eta_{n}^{(i)}U_n).$$
\end{proof}

\begin{corollary}\label{cor:limit}
If (\ref{eq:nonstrict}) is consistent, then a solution
of (\ref{eq:nonstrict}) is
$$( \eta_{1}^{(i)} / \eta_{0}^{(i)} , \ldots ,
 \eta_{n}^{(i)} / \eta_{0}^{(i)} ),$$
where $\eta_{0}^{(i)} U_0 + \cdots + \eta_{n}^{(i)} U_n$ is
a divisor of $R_m$.
\end{corollary}

\subsection{The algorithm modulo $u$-resultant construction}\label{sec:algorithm-modulo}

We noticed in the proof of the Lemma~\ref{le:u-res-m} that factors of $R_m$
correspond to hyperplanes in $(n+1)$-dimensional affine space of coordinates
$U_0, \ldots , U_n$.
Namely, $\eta_{0}^{(i)}U_0 + \cdots + \eta_{n}^{(i)}U_n=0$, $1 \le i \le r$
are equations for the hyperplanes.
Then $(\eta_{0}^{(i)}, \ldots , \eta_{n}^{(i)})$ are vectors orthogonal
to the corresponding hyperplanes.
Obviously, the orthogonal vectors are collinear to gradients, i.e.,
vectors of the kind
$$\biggl( \frac{\partial R_m}{\partial U_0}(x), \ldots ,
\frac{\partial R_m}{\partial U_n}(x) \biggr),$$
for $x$ taken on all hyperplanes (provided that these gradients are nonzero,
see Section~\ref{sec:zero-grad}).

We reduced the consistency problem for (\ref{eq:nonstrict}) to the following:
compute gradients of the polynomial $R_m$ at the points
$x^{(1)}, \ldots , x^{(r)}$, where $x^{(i)}$ belongs to the
hyperplane $\eta_{0}^{(i)}U_0 + \cdots + \eta_{n}^{(i)}U_n=0$,
and does not belong to any other hyperplane.
We shall address the problem of constructing $R_m$ in Section~\ref{sec:computing-u} below.
Now let us describe the rest of the algorithm.

We shall obtain points $x^{(1)}, \ldots , x^{(r)}$ as the points
of intersections of a ``generic'' straight line in ${\mathbb R}^{n+1}$ with $\{ R_m=0 \}$.
Intuitively, {\it almost every} straight line intersects each member of
our family of hyperplanes at exactly one point, and all points are distinct.

Suppose that we have found the generic line in a parametric
form $\alpha t + \beta$, where $t$ is a new variable, and
$\alpha, \beta \in {\mathbb R}^{n+1}$.
Thus, a point on the line has coordinates $\alpha_0 t' + \beta_0,
\ldots , \alpha_n t' + \beta_n$ for a value $t' \in {\mathbb R}$ of the parameter $t$.
The roots of
$$R_m(\alpha t + \beta)=0$$
(w.r.t. the variable $t$) correspond to points of intersection of the line with the hyperplanes.
Now Corollary~\ref{cor:limit} implies that if (\ref{eq:nonstrict}) is consistent, then
for a root $t'$ of $R_m( \alpha t + \beta)=0$ the point
$$\biggl( \frac{\partial R_m /\partial U_1}{\partial R_m / \partial U_0}
(\alpha t' + \beta), \ldots , \frac{\partial R_m / \partial U_n}
{\partial R_m / \partial U_0} (\alpha t' + \beta) \biggr)$$
is a solution to (\ref{eq:nonstrict}).

Denoting
$$\frac{\partial R_m / \partial U_j}{\partial R_m /
\partial U_0}(\alpha t + \beta)= y_j(\alpha t + \beta),\> 1 \le j \le n,$$
we get:
\medskip

\noindent {\em The system (\ref{eq:nonstrict}) is consistent if and only if there exists
$t' \in {\bf R}$ such that
$$R_m(\alpha t' + \beta)=0\quad {\rm and}$$
\begin{equation}\label{eq:one-variable}
f_i \bigl( y_1(\alpha t' + \beta), \ldots , y_n(\alpha t' + \beta) \bigr)
 \ge 0
 \end{equation}
$${\rm for}\> {\rm all}\> 1 \le i \le k.$$}
\medskip

The system (\ref{eq:one-variable}) is a system of polynomial inequalities in {\it one} variable.
We can check its consistency by the algorithm of Ben-Or, Kozen and Reif.

\subsection{Case of a zero gradient}\label{sec:zero-grad}

In Section~\ref{sec:algorithm-modulo} we used the gradients of $R_m$ at points $x$ of intersection
of $\{ R_m=0 \}$ with a generic line, i.e., vectors of the kind
$$\biggl( \frac{\partial R_m}{\partial U_0}(x), \ldots ,
\frac{\partial R_m}{\partial U_n}(x) \biggr),$$
silently assuming that these vectors are nonzero.
This is a simplification of the actual procedure, since the gradient
{\it may be} zero at a chosen point.
This happens exactly when a factor
$\eta_{0}^{(i)}U_0 + \cdots + \eta_{n}^{(i)}U_n$ occurs in $R_m$ with
a power greater than 1.
In terms of $R_m(\alpha t + \beta) $ it means that a root $t'$,
corresponding to an intersection of the straight line with $\{ R_m=0 \}$,
is {\it multiple}.

\begin{lemma}
$$ \eta_{j}^{(i)}= \frac{\partial^p} {\partial s \partial^{p-1} t}
R_m(\alpha t + se_j + \beta) \Big\arrowvert_{t=t',s=0}$$
for a certain $p,\> 1 \le p \le r$, where
$e_j=(0, \ldots , 0,1,0, \ldots , 0)$ (1 occurs at $j$th position).
\end{lemma}

\begin{proof}
By direct computation.
\end{proof}

\subsection{Computing a generic straight line in ${\mathbb R}^{n+1}$}

Let
$$S_1= \{ (1,i,i^2, \ldots ,i^n):\quad 1 \le i \le nr+1 \};$$
$$S_2= \{ (1,i,i^2, \ldots ,i^n):\quad 1 \le i \le nr(r-1)/2 +1 \}.$$

\begin{lemma}
There exist $\alpha \in S_1$ and $\beta \in S_2$
such that:
\medskip

\noindent (1)\quad $\{ \alpha t + \beta \}$ intersects every
$\{ \eta_{0}^{(i)}U_0 + \cdots + \eta_{n}^{(i)}U_n=0 \} $ at exactly
one point, say, $\alpha t^{(i)} + \beta$;
\medskip

\noindent (2)\quad if $t^{(i)}=t^{(j)}$ for $i \neq j$ then
$$\frac{ \eta_{l}^{(i)}}{\eta_{0}^{(i)} } =
\frac{ \eta_{l}^{(j)}}{\eta_{0}^{(j)} }$$
for all $1 \le l \le n$.
\end{lemma}

\begin{proof}
Each subset of $n+1$ vectors from $S_1$
or from $S_2$ is linearly
independent (they form a famous {\it Vandermonde matrix}).
Then at least one $\alpha \in S_1$ does not belong to
$$\{ R_m =0 \} = \bigcup_{1 \le i \le r} \{  \eta_{0}^{(i)}U_0 + \cdots +
\eta_{n}^{(i)}U_n \}.$$

Indeed, let {\it all} points $\alpha^{(i)}=(1,i,i^2, \ldots ,i^n) \in S_1$ belong to $\{ R_m=0 \}$.
Then, denoting $\eta^{(i)}=( \eta_{0}^{(i)}, \ldots , \eta_{n}^{(i)})$, we get:
$$\alpha^{(1)} \eta^{(1)}=0\quad {\rm or}\quad  \alpha^{(1)} \eta^{(2)}=0\quad
{\rm or}\quad \cdots\quad  {\rm or}\quad \alpha^{(1)} \eta^{(r)}=0$$
$${\rm and}$$
$$\alpha^{(2)} \eta^{(1)}=0\quad {\rm or}\quad \alpha^{(2)} \eta^{(2)}=0\quad
{\rm or}\quad \cdots\quad {\rm or}\quad \alpha^{(2)} \eta^{(r)}=0$$
$${\rm and}$$
$$\cdots\quad \cdots\quad \cdots$$
$$\alpha^{(nr+1)} \eta^{(1)}=0\quad {\rm or}\quad \alpha^{(nr+1)} \eta^{(2)}=0\quad
{\rm or}\quad \cdots\quad {\rm or}\quad \alpha^{(nr+1)} \eta^{(r)}=0.$$

It follows that there exist $i_0$ and pair-wise distinct $j_1, \ldots , j_{n+1}$ such that
$$
\alpha^{(j_1)} \eta^{(i_0)}=0\quad {\rm and}\quad \alpha^{(j_2)} \eta^{(i_0)}=0\quad {\rm and}\quad
\ldots\quad {\rm and}\quad \alpha^{(j_{n+1})} \eta^{(i_0)}=0,
$$
i.e., vectors $\alpha^{(j_1)}, \ldots , \alpha^{(j_{n+1})}$ are linearly
dependent. We have got a contradiction.

It follows that for any $\beta$, (1) is satisfied.
Now let us choose $\beta$ so that (2) would be satisfied.

Let $H_1 \subset {\mathbb R}^{n+1}$ be the set of all points contained
in two distinct hyperplanes.
For any two distinct hyperplanes $P_1, P_2$ the dimension
$\dim (P_1 \cap P_2)= n-1$ (recall that each hyperplane passes through
the origin).
The number of all pairs of hyperplanes is $\binom{r}{2}= r(r-1)/2$.
Hence $H_1$ is contained in a union of $r(r-1)/2$
planes of the dimension $n-1$.
It follows that the set
$$H_2= \{ \beta \in {\mathbb R}^{n+1}:\quad \alpha t + \beta \in H_1\>
{\rm for}\> {\rm some}\> t \}$$
is contained in a union of $r(r-1)/2$ planes of the dimension $n$ (i.e.,
hyperplanes).
Reasoning as before, we get that there exists $\beta \in S_2$
such that $\beta \not\in H_2$.
\end{proof}

\subsection{Computing $u$-resultant}\label{sec:computing-u}

In this section we shall discuss briefly the problem of computing
the $u$-resultant.
It's more convenient to do that for a general system of homogeneous
equations (rather than for the concrete system (\ref{eq:homogen}).
The construction will be ``approximately correct'', that is, we shall ignore
some technical details.

We start with a system of homogeneous equations
\begin{equation}\label{eq:general}
g_1= \cdots = g_n=0
\end{equation}
in variables $X_0,X_1, \ldots , X_n$.
Let $\deg(g_i)=d_i$ and $D= \sum_{1 \le i \le n} d_i$.
Introduce new variables $U_0,U_1, \ldots , U_n$ and
construct a new polynomial
$$g_0 \equiv U_0X_0 + U_1X_1 + \cdots +U_nX_n.$$
The degree $\deg (g_0)=d_0=1$.

Observe that each homogeneous polynomial $h$ can be uniquely defined by
a vector of its coefficients (that is, we must fix a certain order of all
possible monomials and take coefficients in that order).
Thus, we can view the set of all polynomials of degree less or equal to $d$
as a vector space of the dimension $\dim = \binom{n+d}{n}$.
Identically zero polynomial plays a role of the zero in that space.

Let ${\mathcal H}_i$ be the vector space of all homogeneous polynomials of degree
$D-d_i$ ($0 \le i \le n$), and ${\mathcal H}$ be the vector space of all
polynomials of the degree $D$.

Consider the direct sum ${\mathcal H}_0 \times \cdots \times {\mathcal H}_n$
which is again the a vector space.
Its elements are vectors of spaces ${\mathcal H}_i$, joined end-to-end.
Observe that
$$\dim ({\mathcal H}_0 \times \cdots \times {\mathcal H}_n)=
\sum_i \binom{n+D-d_i}{n};\quad \dim ({\mathcal H})= \binom{n+D}{n}.$$

Define the linear map from ${\mathcal H}_0 \times \cdots \times {\mathcal H}_n$
to ${\mathcal H}$ by the following formula:
$$(h_0, h_1, \ldots , h_n) \longrightarrow g_0h_0 + \cdots + g_nh_n.$$
The matrix $M$ of this map is of the order
$$\binom{n+D}{ n} \times \sum_i \binom{n+D-d_i}{n}.$$

\begin{lemma}

\noindent (1)\quad The system (\ref{eq:general}) has a finite number of roots if and only if
at least one of the $\binom{n+D}{n} \times \binom{n+D}{ n}$
minors of $M$ is nonsingular;
\medskip

\noindent (2)\quad If $M'$ is a maximal nonsingular minor of $M$, then
$$\det(M')= \prod_i (\xi_{0}^{(i)}U_0 + \cdots +
\xi_{n}^{(i)}U_n), $$
where $ (\xi_{0}^{(i)}: \cdots : \xi_{n}^{(i)})$ are all the roots of
the system (\ref{eq:general}).
\end{lemma}

\begin{proof}[Hint for the proof]
Let (\ref{eq:general}) have finite number of roots
and let $ (\xi_{0}^{(i)}: \cdots : \xi_{n}^{(i)}) $ be one of them.
Then, by the definition of the map $M$,
for any vector $h_0, h_1, \ldots h_n$ the value of the homogeneous
polynomial $M(h_0, \ldots , h_n)$ at the point
$ (\xi_{0}^{(i)}: \cdots : \xi_{n}^{(i)}) $ is zero.
Hence for the specializations of $U_0, \ldots , U_n$ for which
$\xi_{0}^{(i)}U_0+ \cdots + \xi_{n}^{(i)}U_n=0$, the map {\it is not
an isomorphism}.
Indeed, for this specialization every image under the map is a
homogeneous polynomial having $ (\xi_{0}^{(i)}: \cdots : \xi_{n}^{(i)}) $
as a root.
Not every polynomial of the degree $D$ has this property.

It follows that all maximal minors of $M$ (for the specialization of
$U_0, \ldots , U_n$) vanish, in particular, $\det (M')$ vanishes.
In other words: polynomial $\det (M')$ vanishes on the hyperplane
$\{ \xi_{0}^{(i)}U_0 + \cdots + \xi_{n}^{(i)}U_n=0 \}$.
It follows that  $\xi_{0}^{(i)}U_0 + \cdots + \xi_{n}^{(i)}U_n$
divides $\det (M')$.
\end{proof}

\subsection{Complexity of deciding consistency}

A straightforward observation shows that the complexity of the described algorithm is $(kd)^{O(n)}M^{O(1)}$
for the Turing machine model and $(kd)^{O(n)}$ for the real number model.
Here the term $kd$ appears as the degree of polynomial $g(\eps)$ in Section~\ref{sec:epsilon}.

\section{Quantifier elimination}

Here we briefly sketch main ideas behind the quantifier elimination algorithm.
To make formulas simpler we will work in the real number model.

\subsection{Formulation of the problem}

We first consider the problem of elimination one block of existential quantifiers, i.e., representing
a set
$$
\left\{ \exists X_{n_1} \exists X_{n_1+1} \cdots \exists X_n \left(
f_1 (X_1, \ldots , X_n) \ge 0 , \ldots , f_k (X_1, \ldots , X_n) \ge 0 \right) \right\}
\subset {\mathbb R}^{n_1-1}
$$
by a Boolean combination of polynomial equations and inequalities.

Changing, for convenience, the notations we write:
$$
X=(X_{n_1}, \ldots ,X_n),\quad Y=(X_1, \ldots , X_{n_1-1}),
$$
$$
F(X_1, \ldots , X_n)=f_1 (X_1, \ldots , X_n) \ge 0 , \ldots , f_k (X_1, \ldots , X_n) \ge 0,
$$
thus the input formula becomes $\exists X\ F(Y)$.

\subsection{Main idea}

Consider free variables $Y$ as {\em parameters} and try to execute the
decidability algorithm for a parametric formula (i.e., with variable coefficients).

Decidability algorithm from Section~\ref{sec:systems} uses only arithmetic operations $\pm, \times, \div$,
and comparisons over real numbers.
Arithmetic can be easily performed parametrically (it is arithmetic of polynomials),
while comparisons require branching.

The parametric algorithm can be represented by {\em algebraic decision tree}.

The tree is described as follows.
Root of the tree is associated with an arithmetic operation on some variables in the variable vector $Y$.
Proceeding by induction, at each vertex of the tree, except leaves, an arithmetic operation
is performed between two elements of the set consisting of variables in $Y$ and the results of operations
performed along the branch leading from the root to this vertex.
As a result, this vertex is associated with a polynomial in variables in $Y$, which is
a composition of arithmetic operations performed along the branch.

Upon a {\em specialization} of $Y$, the polynomial at each vertex is evaluated, and the value is compared to zero.
Depending on the result of comparison, one of the three branches is chosen,
thereby determining the next vertex, where the next polynomial evaluation will be made,
and so on, until a leaf is reached.

Each leaf is assigned the value {\bf True} or {\bf False}.
All specializations of $Y$, arriving at a leaf marked {\bf True}, correspond to closed formulas that are true.
Similar with {\bf False}.

\subsection{Elimination algorithm}

\begin{itemize}
\item
Use the algorithm from Section~\ref{sec:systems} parametrically to obtain a decision tree
with the variable input $Y$.
\item
Determine all branches from the root to leaves marked {\bf True}.
\item
Take conjunction of the inequalities along each such branch:
$$\bigwedge_{i \in {\rm branch}}h_i(Y) \sigma_i 0,$$
where $\sigma_i \in \{ <,>,= \}$.
\item
$$\exists X\ F(X,Y) \Leftrightarrow \bigvee_{\rm branches}\quad
\bigwedge_{i \in {\rm branch}}h_i(Y) \sigma_i 0$$
\end{itemize}

\subsection{Complexity of elimination algorithm}

The complexity of the consistency algorithm from Section~\ref{sec:systems} coincides with {\em height} of the tree
(the length of the longest branch from the root to a leaf).
From the first glance, this leads to a bad complexity for elimination algorithm,
because there are too many branches (hence, leaves) in the tree.
Indeed, the height equals to complexity of consistency, which is $(kd)^{O(n)}$.
It follows that combinatorially there may be as much as $3^{(kd)^{O(n)}}$ leaves
(recall that we take the union of leaves marked {\bf True} in the elimination algorithm).
However, we will now show that most branches are not followed under any specialization of $Y$, i.e.,
most leaves correspond to Boolean formulas defining $\emptyset$.

Let us first generalize Definition~\ref{def:css} to the case of many variables.

Let $f_1, \ldots , f_k \in \Real[X_1, \ldots , X_n]$ with degrees
$\deg (f_i)< d$ for all $i,\> 1 \le i \le k$.

\begin{definition}
Consistent sign assignment for $f_1, \ldots , f_k$ is a string
$$\sigma = ( \sigma_1, \ldots , \sigma_k ),$$
where $\sigma_i \in \{ >,<,= \}$, such that the system
$f_1 \sigma_1 0, \ldots , f_k \sigma_k 0$ has a root in $\Real^n$.
\end{definition}

\begin{lemma}\label{le:sa}
The number of distinct consistent sign assignments for $f_1, \ldots , f_k$
is at most $(kd)^{O(n)}$, where $\deg (f_i) <d$ for all $i \in \{ 1, \ldots ,k \}$.
\end{lemma}

\begin{proof}
Choose in every $\{ f_1 \sigma_1 0, \ldots , f_k \sigma_k 0 \}$ one arbitrary point.
There exists a positive $\eps \in \Real$ such that for every chosen point $x$ and
every $i \in \{ 1, \ldots ,k \}$ the inequality $f_i(x)>0$ implies $f_i(x)> \eps$, and
$f_i(x)< 0$ implies $f_i(x)< -\eps$.
It is easy to prove that the number of distinct consistent sign assignments does not exceed
the number of connected components of the semialgebraic set
$$
S= \left\{ h= \prod_{1 \le i \le k} (f_i+ \eps)^2(f_i- \eps)^2 >0 \right\}.
$$
In turn, according to the famous Thom-Milnor bound on homologies of semialgebraic sets,
this number does not exceed $(kd)^{O(n)}$.

Alternatively, for a small enough positive $\delta < \eps$, the number of connected components of $S$
does not exceed the number of connected components of $\{ h= \delta \}$.
A careful examination of the consistency algorithm from Section~\ref{sec:systems} shows that, being applied to
the equality $h= \delta$, it will produce among all of its roots
$$(\eta_{1}^{(i)}/\eta_{0}^{(i)}, \ldots , \eta_{n}^{(i)}/\eta_{0}^{(i)})$$
a representative on each connected component of $\{ h= \delta \}$.
The upper bound on the number of roots follows from the complexity of the consistency algorithm.
\end{proof}

Carefully examining the decidability algorithm we can see that the number
of distinct polynomials $h_i$ in the tree is essentially the same as
in one branch, i.e., $(kd)^{O(n)}$.
From the complexity estimate we know that their degrees are $(kd)^{O(n)}$.
Then, by Lemma~\ref{le:sa}, the number of distinct non-empty sets associated
with leaves is $(kd)^{O(n^2)}$.
We conclude that the complexity of $\exists$-elimination is $(sd)^{O(n^2)}$.

\subsection{Universal quantifier and quantifier alternation}

To eliminate the universal quantifier in $\forall X\ F(X,Y)$ the algorithm first re-writes
this formula in an equivalent form $\neg \exists X \neg F(X,Y)$, then eliminates $\exists$
in $\exists X \neg F(X,Y)$,
as above, and finally re-writes the resulting Boolean combination of equations and inequalities
without the negation symbol.

Finally, consider the general case of a formula
$$(Q_1 X^{(1)}) \cdots (Q_q X^{(q)}) P(Y, X^{(1)}, \ldots , X^{(q)}),$$
where
\begin{itemize}
\item
$Q_1, \ldots , Q_q$ are quantifiers $\exists$ or $\forall$;

\item
$X^{(i)}=(X_{1}^{(i)}, \ldots , X_{n_i}^{(i)})$ --- bound
variables;

\item $Y=(Y_1, \ldots , Y_{n_0})$ --- free variables;

\item $P$ --- Boolean combination of $k$ formulas
of the kind $f>0$ or $f=0$, and for every $f$ the degree $\deg (f)<d$.
\end{itemize}

Set $n= n_0+n_1+ \cdots + n_q$.

The algorithm begins with eliminating the most ``internal'' quantifier $Q_q$
in the formula
$$(Q_q X^{(q)}) P(Y, X^{(1)}, \ldots , X^{(q)}),$$
obtaining a Boolean combination $P'(Y, X^{(1)}, \ldots , X^{(q-1)})$, and then continues recursively,
eliminating quantifiers one by one, inside out.

Clearly, the complexity bound of this procedure is $(kd)^{n^{O(q)}}$

Using finer technical tools one can construct quantifier elimination
algorithm having the complexity
$$(kd)^{\prod_{0 \le i \le q}O(n_i)}.$$

\end{document}